\documentclass{amsart}

\usepackage{amsmath, amsthm, mathtools, amssymb}
\usepackage{hyperref}

\newtheorem{thm}{Theorem}[section]
\newtheorem{lem}[thm]{Lemma}

\theoremstyle{definition}
\newtheorem{defn}[thm]{Definition}

\theoremstyle{remark}
\newtheorem{rem}[thm]{Remark}

\newcommand*{\braces}[1]{\left\lbrace #1 \right\rbrace}
\newcommand*{\setof}{\braces}
\newcommand*{\parens}[1]{\left( #1 \right)}

\newcommand*{\deftobe}{\mathrel{\coloneqq}}

\newcommand{\bndr}{\partial}
\DeclareMathOperator{\conv}{conv}
\newcommand{\relint}[1]{\operatorname{int}(#1)}
\newcommand{\intr}{\relint}
\DeclareMathOperator{\aff}{aff}

\DeclareMathOperator{\cone}{cone}

\newcommand{\R}{\mathbb{R}}
\renewcommand{\epsilon}{\varepsilon}

\title{Two characterizations of ellipsoidal cones}

\author[J.~Jer\'onimo-Castro]{%
   Jes\'us Jer\'onimo-Castro
}%
\address[Jes\'us Jer\'onimo-Castro]{%
   Facultad de Ingenier'a \\
   Universidad Aut\'onoma de Quer\'etaro \\
   Cerro de las Campanas s/n, C.P. 76010 \\
   Quer\'etaro, M\'exico
}%

\author[T.~B.~McAllister]{%
   Tyrrell B. McAllister
}%
\address[Tyrrell B. McAllister]{%
   Department of Mathematics \\
   University of Wyoming \\
   Laramie, WY 82071 \\
   USA
}%
\email[Tyrrell B. McAllister]{%
   tmcallis@uwyo.edu
}%

\date{}

\begin{document}

\begin{abstract}
   We give two characterizations of cones over ellipsoids.  Let
   $C$ be a closed pointed convex linear cone in a
   finite-dimensional real vector space.  We show that $C$ is a
   cone over an ellipsoid if and only if the affine span of $\bndr
   C \cap \bndr(a - C)$ has dimension $\dim(C) - 1$ for every
   point $a$ in the relative interior of $C$.  We also show that
   $C$ is a cone over an ellipsoid if and only if every bounded
   section of $C$ by an affine hyperplane is centrally symmetric.
\end{abstract}

\maketitle


\section{Introduction}

The following fact is an easy exercise in geometry: If $E$ is an
$n$-dimensional solid ellipsoid and $a$ is a vector, then $\bndr E
\cap \bndr(a - E)$ is contained in an affine hyperplane unless $E
= a - E$.  A far more difficult result due to P.~R.~Goodey and
M.~M.~Woodcock \cite{GoodWood1981} shows that this property
suffices to characterize ellipsoids: Ellipsoids are the only
convex bodies whose boundaries have a ``flat'' intersection with
all non-coincident translates of their negatives.  Another
characterization of ellipsoids is the famous False Centre Theorem
of P.~W.~Aitchison, C.~M.~Petty, and C.~A.~Rogers \cite{APR1971}.
This result, first conjectured by Rogers in \cite{Rogers1965},
states that a convex body $K$ ($\dim K \ge 3$) is an ellipsoid if
there is a point $p$, not a center of symmetry of $K$, such that
every section of $K$ by a hyperplane through $p$ is centrally
symmetric.  Gruber and \'Odor \cite{GO1999} exhibit another sense
in which sufficiently symmetric convex bodies must be ellipsoids:
If $K$ is a convex body such that the cone over $K$ from every
point outside of $K$ is symmetric about some axis, then $K$ is an
ellipsoid.

We prove analogous results for convex cones.  Let us begin by
fixing our notation and terminology.  Let $V$ be a
finite-di\-men\-sion\-al real vector space.  A \emph{convex linear
cone} in $V$ is a nonempty convex subset $C \subset V$ such that
$a \in C$ and $\lambda \ge 0$ implies that $\lambda a \in C$.  The
cone $C$ is \emph{pointed} if there exists a hyperplane $H$ in $V$
such that $H \cap C = \setof{0}$.  Henceforth, we simply write
``cone'' for ``closed pointed convex linear cone'', unless
otherwise specified.

Given a convex subset $K \subset V$, let $\aff(K)$ denote the
affine span of $K$.  We write $\relint{K}$ and $\bndr K$ for the
interior and boundary, respectively, of $K$ relative to $\aff(K)$
under the subspace topology.  The \emph{cone over} $K$, denoted
$\cone(K)$, is the intersection of all cones containing $K$.  A
\emph{section} of $K$ is a $(\dim K - 1)$-dimensional intersection
of $K$ with an affine hyperplane.  A
\emph{\textup{(}solid\textup{)} ellipsoid} in $V$ is the image of
the closed unit ball in some Euclidean vector space $E$ under an
affine map $E \to V$.  An \emph{ellipse} is a $2$-dimensional
ellipsoid.  A cone $C$ is \emph{ellipsoidal} if some section of
$C$ is an ellipsoid.

\begin{defn}[FBI and CSS Cones]
   Let $C \subset V$ be a cone.  We say that $C$ satisfies the
   \emph{flat boundary intersections \textup{(}FBI\textup{)}
   property} if, for each $a \in \relint C$, the affine span of
   $\bndr C \cap \bndr(a - C)$ has dimension $\dim(C) - 1$.  We
   say that $C$ satisfies the \emph{centrally symmetric sections
   \textup{(}CSS\textup{)} property} if every bounded section of
   $C$ is centrally symmetric.  We call a cone with the FBI
   (respectively, CSS) property an \emph{FBI
   \textup{(respectively,} CSS\textup{)} cone}.
\end{defn}

It is easy to check that every finite-dimensional ellipsoidal cone
satisfies the FBI property: Let $C$ be an $(n+1)$-dimensional cone
over an ellipsoid.  Then there exists a linear system $x = (x_{0},
\dotsc, x_{n})$ of coordinates on the linear span of $C$ such that
$C$ is the set of solutions to
\begin{equation*}
   x_{0}^{2} \ge x_{1}^{2} + \dotsb + x_{n}^{2}, \qquad x_{0}
   \ge 0.
\end{equation*}
Fix a point $a = (a_{0},\dotsc,a_{n}) \in \relint C$, and let
$\bar{a} = (-a_{0}, a_{1}, \dotsc, a_{n})$.  Then $\bndr C \cap
\bndr(a - C)$ is contained in the affine hyperplane of solutions
to the linear equation $\bar{a} \cdot x = \frac{1}{2} \bar{a}
\cdot a$.

Our first main result is that the only finite-dimensional cones
satisfying the FBI property are the ellipsoidal cones.

\begin{thm}[proved on p.~\pageref{proof:FBIimpliesEllipsoidal}]
   \label{thm:FBIimpliesEllipsoidal}
   A cone $C \subset V$ is an FBI cone if and only if $C$ is an
   ellipsoidal cone.
\end{thm}

\begin{rem}
   We make a short digression to note that Theorem
   \ref{thm:FBIimpliesEllipsoidal} provides a natural motivation
   for the use of a Lorentzian inner product in special
   relativity.  One way to develop special relativity begins as
   follows: Let $M$ be a $4$-dimensional real affine space of
   \emph{events} with associated vector space $V \cong \R^{4}$.
   Fix a $4$-dimensional cone $C \subset V$, called the
   \emph{light cone}.  At this stage, we do not yet assume that
   $C$ is ellipsoidal.  Given a pair of events $a, b \in M$ with
   $b - a \in \intr C$, define the \emph{inertial reference frame}
   $\smash{\overline{ab}}$ to be the set of affine lines in $M$
   parallel to the affine span of $\setof{a,b}$.
   
   Traditionally, one proceeds by assuming that $\bndr C$ is (one
   half of) the null cone of a Lorentzian inner product on $V$.
   This is equivalent to assuming that $C$ is ellipsoidal.  We
   then associate to a given inertial reference frame
   $\smash{\overline{ab}}$ a decomposition of $V$ into a direct
   sum of a $3$-dimensional ``space'' summand, parallel to the
   affine span of $\bndr (a + C) \cap \bndr(b - C)$, and a
   $1$-dimensional ``time'' summand, spanned by $b - a$.
   
   However, instead of assuming that $C$ is ellipsoidal, we may
   instead proceed from the assumption that every choice of
   inertial reference frame yields a decomposition of $V$ into a
   $3$-dimensional space summand and a $1$-dimensional time
   summand in the manner just described.  In other words, we may
   assume that the affine span of $\bndr (a + C) \cap \bndr(b -
   C)$ is $3$-dimensional.  It then follows from Theorem
   \ref{thm:FBIimpliesEllipsoidal} that the light cone is
   ellipsoidal, so that a Lorentzian inner product arises
   naturally.  Thus we derive a Lorentzian inner product on
   spacetime from the phenomenologically immediate datum that
   space is 3-dimensional.
\end{rem}

Since every section of an ellipsoidal cone is an ellipsoid, it is
clear that ellipsoidal cones are CSS cones.  Our second main
result is that the ellipsoidal cones are precisely the CSS cones.

\begin{thm}[proved on
p.~\pageref{proof:FalseCenterTheoremForCones}]
   \label{thm:CSSimpliesEllipsoidal}
   A cone $C \subset V$ is a CSS cone if and only if $C$ is an
   ellipsoidal cone.
\end{thm}

We call Theorem \ref{thm:CSSimpliesEllipsoidal} the False Centre
Theorem for Cones, by analogy with the famous False Centre Theorem
characterizing ellipsoids \cite{APR1971}.  The special case of
Theorem \ref{thm:CSSimpliesEllipsoidal} in which $C$ is a
$3$-dimensional cone was proved in \cite{Olov1941}.  A more-recent
independent proof of this $3$-dimensional case appeared in the
B.S.\ thesis of Efr\'en Morales-Amaya \cite{Mor1998}.  Solomon
\cite{Sol2012} shows that any complete connected $C^{2}$ surface
in $\R^{3}$ whose sections are all centrally symmetric ovals is
either a cylinder or a quadric.  We note that neither Theorem
\ref{thm:FBIimpliesEllipsoidal} nor Theorem
\ref{thm:CSSimpliesEllipsoidal} relies on any smoothness
assumptions on the boundary of the cone, though convexity will be
crucial for our proofs.

The outline of our argument is as follows.  In Section
\ref{sec:FBICones}, we use the previously established result that
$3$-dimensional CSS cones are ellipsoidal (\emph{i.e.}, the
$3$-di\-men\-sion\-al case of Theorem
\ref{thm:CSSimpliesEllipsoidal}) to prove that $n$-dimensional FBI
cones are ellipsoidal (Theorem \ref{thm:FBIimpliesEllipsoidal}).
Then, in Section \ref{sec:CSSCones} we use Theorem
\ref{thm:FBIimpliesEllipsoidal} to prove Theorem
\ref{thm:CSSimpliesEllipsoidal} for cones of arbitrary finite
dimension.

%

\section{FBI cones are ellipsoidal cones}
\label{sec:FBICones}

We begin with a few straightforward lemmas about the intersection
of the boundary of a cone $C$ with the boundary of a translation
of $-C$ by a vector in the interior of $C$.

\begin{lem}
   \label{lem:BoundaryIntersectionsAreCentrallySymmetric}
   Suppose that $C \subset V$ is a cone and that $a \in \intr C$.
   Then $\bndr C \cap \bndr(a - C)$ is centrally symmetric about
   $\frac{1}{2}a$.
\end{lem}

\begin{proof}
   The translation $\bndr C \cap \bndr(a - C) - \frac{1}{2}a =
   \bndr (C - \frac{1}{2}a) \cap \bndr(-C + \frac{1}{2}a)$ is
   centrally symmetric about the origin, so the original
   intersection is centrally symmetric about $\frac{1}{2}a$.
\end{proof}

\begin{lem}
   \label{lem:BoundaryIntersectionsAreSpheres}
   Suppose that $C \subset V$ is a cone and that $a \in \intr C$.
   Let $\Gamma \deftobe \bndr C \cap \bndr(a - C)$ and let $S$ be
   a bounded section of $C$.  Then every point on $\Gamma$ (resp.\
   $\bndr S$) is a unique scalar multiple of a unique point on
   $\bndr S$ (resp.\ $\Gamma$).  Moreover, this correspondence
   $\bndr S \leftrightarrow \Gamma$ is a homeomorphism, so that
   $\Gamma$ is homeomorphic to an $n$-sphere.
\end{lem}
\begin{proof}
   Fix a bounded section $S$ of $C$.  Then every point on $\bndr C
   \setminus \setof{0}$ is a unique scalar multiple of a unique
   point on $\bndr S$.  In particular, we have a map $\Gamma \to
   \bndr S$.

   To establish the converse correspondence $\bndr S \to \Gamma$,
   fix $\lambda > 0$ so that $\lambda a \in S$.  Given $x \in
   \bndr S$, let $r(x)$ be the opposite endpoint of the chord of
   $S$ through $\lambda a$ starting at $x$.  Let $\mu_{x} \in
   (0,1/\lambda)$ be such that $\lambda a = \lambda\mu_{x}x +
   \parens{1 - \lambda\mu_{x}} r(x)$.  On the one hand, $\mu_{x} >
   0$ and $x \in \bndr C$ imply that $\mu_{x} x \in \bndr C$.  On
   the other hand, $\mu_{x} < 1/\lambda$, $r(x) \in \bndr C$, and
   $\mu_{x} x = a - (1/\lambda - \mu_{x}) r(x)$ together imply
   that $\mu_{x} x \in \bndr(a - C)$.  Hence, $\mu_{x} x \in
   \Gamma$.  Since $a - C$ is a pointed affine cone containing the
   origin in its interior, $\mu_{x} x$ is the unique multiple of
   $x$ on $\bndr(a - C)$.  Finally, observe that $x \mapsto
   \mu_{x} x$ is a continuous map with a continuous inverse,
   establishing that $\bndr S$ and $\Gamma$ are homeomorphic.
\end{proof}

\begin{lem}
   \label{lem:2DConvexBodyContainsParallelogram}%
   Let $K$ be a $2$-dimensional convex body.  Then $K$ contains an
   inscribed parallelogram with vertices in $\bndr K$.
\end{lem}
\begin{proof}
   It is easy to construct such a parallelogram using the
   intermediate value theorem and continuity of the boundary of
   $K$.  For example, consider the family $F$ of chords
   perpendicular to a fixed diameter of $K$.  Choose two chords
   $\chi_{1}, \chi_{2} \in F$ that are of equal length and that
   are on opposite sides of a chord of maximum length in $F$.
   Then $P \deftobe \conv(\chi_{1} \cup \chi_{2})$ is an incribed
   parallelogram in $K$.  Indeed, the stronger claim that $K$
   contains an inscribed \emph{square} is a classical result; see,
   \emph{e.g.}, \cite{Emc1913}.
\end{proof}

We remark that the natural generalization of Lemma
\ref{lem:2DConvexBodyContainsParallelogram} to higher dimensions
does not hold.  There exist convex bodies in dimension $n \ge 5$
that do not contain inscribed parallelepipeds \cite{HMS1997}.

We will also appeal to the following classical characterization of
ellipsoids due to Brunn \cite{bru1889}; see also
\protect{\cite[Lemma 16.12, p.91]{Bus1955}}.

\begin{thm}\label{thm:BrunnsTheorem}
   Let $n \ge 3$ and let $K$ be an $n$-dimensional convex body
   such that every section of $K$ is an ellipsoid.  Then $K$
   itself is an ellipsoid.
\end{thm}

The key additional result on which the proof of Theorem
\ref{thm:FBIimpliesEllipsoidal} depends is the following
characterization of $3$-dimensional cones over ellipses,
originally due to Olovjanischnikoff \cite{Olov1941}.  See also
\cite{Mor1998}.

\begin{thm}[False Centre Theorem for $3$-dimensional cones]
   \label{thm:FalseCenterTheoremFor3DCones}
   If $C \subset V$ is a $3$-dimensional CSS cone, then $C$ is a
   cone over an ellipse.
\end{thm}

We are now ready to prove the main result of this section, that
all finite-dimensional FBI cones are ellipsoidal cones
(Theorem~\ref{thm:FBIimpliesEllipsoidal}).

\begin{proof}[Proof of Theorem~\ref{thm:FBIimpliesEllipsoidal}
   (stated on p.~\pageref{thm:FBIimpliesEllipsoidal})]
   \label{proof:FBIimpliesEllipsoidal}
   Let $C$ be an $(n+1)$-dimensional FBI cone.  Without loss of
   generality, we suppose that $C$ is full-dimensional.  We begin
   by proving the $n = 2$ case.  The case where $n \ge 3$ will
   then follow by induction.

   Suppose that $\dim(C) = 3$.  Fix a bounded section $S$ of $C$,
   and let $H$ be the affine span of $S$.  By
   Lemma~\ref{lem:2DConvexBodyContainsParallelogram}, there exists
   a parallelogram $P$ with vertices in $\bndr S$.  Let $p$ be the
   intersection of the diagonals of $P$, let $a \deftobe 2p$, and
   let $\Gamma \deftobe \bndr C \cap \bndr(a - C)$.  Observe that
   the vertices of $P$ are contained in $\Gamma$ because $P$ is
   fixed under inversion through $p$.  By hypothesis, $\Gamma$ is
   contained in some plane $H'$, so we also have $P \subset H'$.
   Since $P$ is contained in a unique hyperplane, we have that $H
   = H'$.  By Lemma~\ref{lem:BoundaryIntersectionsAreSpheres},
   $\Gamma$ is a curve homeomorphic to a circle and contained in
   $\bndr S = \bndr C \cap H$.  It follows that $S$ is the convex
   hull of $\Gamma$, which, by
   Lemma~\ref{lem:BoundaryIntersectionsAreCentrallySymmetric}, is
   centrally symmetric.  Therefore, $C$ is a $3$-dimensional CSS
   cone and hence, by
   Theorem~\ref{thm:FalseCenterTheoremFor3DCones}, is a cone over
   an ellipse.

   We proceed by induction.  Suppose now that $\dim(C) = n+1$ for
   $n \ge 3$.  Fix a bounded section $S$ of $C$, let $K$ be a
   section of $S$, let $D \deftobe \cone(K)$, and let $L$ be the
   linear span of $D$.  Fix a point $a$ in the relative interior
   of $D$, and let $\Gamma \deftobe \bndr C \cap \bndr(a - C)$.
   On the one hand, $\bndr D \cap \bndr(a - D)$ is contained in
   $L$.  On the other hand, $\bndr D \cap \bndr(a - D)$ is a
   subset of $\Gamma$, which, since $C$ is an FBI cone, is
   contained in some $n$-dimensional hyperplane $H$.  Note that
   $H$ is not equal to $L$, since $H$ has a bounded intersection
   with $C$ while $L$ does not.  Hence $\bndr D \cap \bndr(a - D)$
   is contained in the intersection of two distinct
   $n$-dimensional hyperplanes, so $\bndr D \cap \bndr(a - D)$ is
   contained in some $(n-1)$-dimensional affine subspace.  That
   is, $D$ is an $n$-dimensional FBI cone and hence is
   ellipsoidal.  In particular, $K$ is an ellipsoid, which, by
   Theorem~\ref{thm:BrunnsTheorem}, implies that $S$ is an
   ellipsoid, proving the claim.
\end{proof}

\section{CSS cones are ellipsoidal cones}
\label{sec:CSSCones}

Our proof of Theorem \ref{thm:FBIimpliesEllipsoidal} relied on the
False Centre Theorem for $3$-di\-men\-sion\-al cones.  It is
natural to ask whether a False Centre Theorem holds for cones of
arbitrary dimension.  So far as we know, such a generalization of
Theorem \ref{thm:FalseCenterTheoremFor3DCones} has not appeared in
the literature.  In this section, we use the FBI characterization
of ellipsoidal cones (proved in Section \ref{sec:FBICones}) to
prove that the CSS property also characterizes ellipsoidal cones
of arbitrary finite dimension.

A well-known result in convexity states that every point in the
interior of a convex body is the centroid of some section of that
body:

\begin{thm}[\cite{Ste1955}; see also \cite{Grue1961}]
   \label{thm:EveryInteriorPointInBodyIsACentroid}
   Let $K \subset V$ be a convex body, and let $p \in \intr K$.
   Then there exists a section $S$ of $K$ such that $p$ is the
   centroid of $S$.
\end{thm}

We will need the analogous result for cones, which we prove using
the above theorem together with a theorem due to Hammer
\cite{Ham1951} bounding the ratio in which the centroid of a
convex body can divide a chord of that body:

\begin{thm}[\cite{Ham1951}]\label{thm:HammersTheorem}
   Let $K$ be an $n$-dimensional convex body, and let $p$ be the
   centroid of $K$.  Then, for each chord $[x, y]$ of $K$ through
   $p$, the convex combination $p = (1 - \mu) x + \mu y$ satisfies
   $\frac{1}{n+1} \le \mu \le \frac{n}{n+1}$.
\end{thm}

\begin{thm}
   \label{thm:EveryInteriorPointInConeIsACentroid}
   Let $C \subset V$ be an $n$-dimensional cone, and let $p \in
   \intr C$.  Then there exists a bounded section $S$ of $C$ such
   that $p$ is the centroid of $S$.
\end{thm}

\begin{proof}
   Fix $\lambda > n+1$, and let $H$ be a hyperplane such that $C
   \cap H$ is bounded and $C \cap (\lambda p - C)$ lies in a
   closed half-space bounded by $H$.  Let $K$ be the intersection
   of $C$ with this closed half-space.  Since $K$ is a convex
   body, there exists a section $S$ of $K$ with centroid $p$ by
   Theorem \ref{thm:EveryInteriorPointInBodyIsACentroid}.

   We claim that $S$ does not intersect $C \cap H$.  Suppose
   otherwise, and let $[x,y]$ be a chord of $S$ through $p$ with
   $x \in \bndr C$ and $y \in C \cap H$.  Let $y' \deftobe x +
   \lambda(p - x)$.  Then $y' = \lambda p - (\lambda - 1) x \in
   \bndr(\lambda p - C)$ is the point where the ray from $x$
   through $p$ meets $\bndr(\lambda p - C)$.  Since $p = \parens{1
   - \frac{1}{\lambda}} x + \frac{1}{\lambda} y'$, and since $y'
   \in [x, y]$, we must have that $p = (1 - \mu)x + \mu y$ for
   some $\mu \le 1/\lambda < \frac{1}{n+1}$.  Therefore, by
   Theorem~\ref{thm:HammersTheorem}, $p$ is not a centroid of $S$,
   a contradiction.

   Since $S$ does not intersect $C \cap H$, it follows that $S$ is
   a bounded section of $C$ with centroid $p$, as desired.
\end{proof}

\begin{lem}
   \label{lem:IntersectionOfBoundariesIsBoundaryOfSection}
   Let $C$ be a cone with $\dim C \ge 2$.  If $p \in \intr C$ is
   the center of symmetry of a bounded section $S$ of $C$, then
   $\bndr S = \bndr C \cap \bndr(2p - C)$.
\end{lem}

\begin{proof}
   Let $\Gamma \deftobe \bndr C \cap \bndr(2p - C)$.  Suppose that
   $\bndr S \ne \Gamma$.  Then, by Lemma
   \ref{lem:BoundaryIntersectionsAreSpheres}, there is a point $x
   \in \bndr S$ such that $\mu x \in \Gamma$ for some $\mu \ne 1$.
   Let $L$ be the $2$-dimensional linear span of $p$ and $x$.  By
   Lemma \ref{lem:BoundaryIntersectionsAreCentrallySymmetric},
   $\bndr S$ and $\Gamma$ are both centrally symmetric about $p$.
   Inversion through $p$ fixes $L$, so we must have that $2p - x
   \in \bndr S \cap L$ and $2p - \mu x \subset \Gamma \cap L$.  By
   applying Lemma \ref{lem:BoundaryIntersectionsAreSpheres} to the
   $2$-dimensional cone $C \cap L$, we find that $2p - x$ and $2p
   - \mu x$ must be scalar multiples of each other.  However, this
   is evidently not the case, since $p$ and $x$ form a basis for
   $L$ and $\mu \ne 1$.  Thus, $\bndr S = \Gamma$, as claimed.
\end{proof}

It is now straightforward to prove Theorem
\ref{thm:CSSimpliesEllipsoidal} as a corollary of the above
results.

\begin{proof}[Proof of Theorem \ref{thm:CSSimpliesEllipsoidal}
   (stated on p.~\pageref{thm:CSSimpliesEllipsoidal})]
   \label{proof:FalseCenterTheoremForCones}
   Suppose that $C$ is a finite-dimensional CSS cone.  Let a point
   $a \in \intr C$ be given, and let $p \deftobe \frac{1}{2}a$.
   By Theorem \ref{thm:EveryInteriorPointInConeIsACentroid}, there
   is a bounded section $S$ of $C$ with centroid $p$.  By
   hypothesis, $S$ is centrally symmetric.  The centroid of a
   centrally symmetric convex body is the center of symmetry of
   the body, so $p$ is the center of symmetry of $S$.  By
   Lemma~\ref{lem:IntersectionOfBoundariesIsBoundaryOfSection},
   $\bndr C \cap \bndr(a - C) = \bndr S \subset \aff(S)$, which
   implies that $C$ is an FBI cone.  Therefore, by
   Theorem~\ref{thm:FBIimpliesEllipsoidal}, $C$ is ellipsoidal.
\end{proof}


\providecommand{\bysame}{\leavevmode\hbox to3em{\hrulefill}\thinspace}
\providecommand{\MR}{\relax\ifhmode\unskip\space\fi MR }
\providecommand{\MRhref}[2]{%
  \href{http://www.ams.org/mathscinet-getitem?mr=#1}{#2}
}
\providecommand{\href}[2]{#2}

\end{document}